\documentclass[letterpaper]{article}

\usepackage{url}
\usepackage{csquotes}
\renewcommand{\mkbegdispquote}[2]{\itshape}
\usepackage{xspace}
\usepackage[group-separator={,}]{siunitx}
\usepackage{booktabs}
\usepackage{amsmath,amssymb,amsfonts,amsthm}
\usepackage{bm}
\usepackage{graphicx}
\usepackage{tikz}
\usepackage{sidecap}

\def\BibTeX{{\rm B\kern-.05em{\sc i\kern-.025em b}\kern-.08em
    T\kern-.1667em\lower.7ex\hbox{E}\kern-.125emX}}

\usepackage{jmlr2e}

\usepackage{caption} 
\usepackage{algorithm}
\usepackage{algorithmic}
\usepackage{mathtools}
\usepackage{newtxtext}
\usepackage[vvarbb]{newtxmath}
\usepackage{caption}
\usepackage{subcaption}
\usepackage{fancyvrb}
\usepackage{arydshln}

\usepackage{multicol}
\usepackage{multirow}
\usepackage{makecell}
\usepackage{rotating}

\newif\ifTODO 
\TODOtrue

\usepackage[color=green]{todonotes}

\usepackage{newfloat}
\usepackage{listings}
\floatstyle{ruled}
\newfloat{listing}{tb}{lst}{}
\floatname{listing}{Listing}

\newcommand{\SDPA}{\textsf{SDPA}\xspace}
\newcommand{\DSDP}{\textsf{DSDP}\xspace}
\newcommand{\PEBBL}{\textsf{PEBBL}\xspace}
\newcommand{\Spectra}{\textsf{Spectra}\xspace}

\newcommand{\field}{\mathbb}
\newcommand{\reals}{\field{R}}
\newcommand{\R}{\reals} 

\newcommand{\abbr}[1][abbrev]{#1.\xspace}

\newcommand{\eg}{\abbr[e.g]}

\newcommand{\ie}{\abbr[i.e]}

\newcommand{\wrt}{\abbr[w.r.t]}
\newcommand{\st}{\mathrm{s.t.}}

\DeclareMathOperator*{\argmin}{argmin}

\DeclarePairedDelimiterX{\innp}[2]{\langle}{\rangle}{#1, #2}

\newtheorem{theorem}{Theorem}
\newtheorem{lemma}[theorem]{Lemma}

\theoremstyle{remark}
\newtheorem{remark}[theorem]{Remark}

\begin{document}

\begin{center}
 {\LARGE QUBO Dual Bounds via SDP Plane Projection Method}
\end{center}

\vspace{8mm}

\textbf{Apostolos Chalkis}\hfill{\ttfamily apostolos.chalkis@quantagonia.com}\\
{\small\emph{Quantagonia}}

\textbf{Thomas Kleinert}\hfill{\ttfamily thomas.kleinert@quantagonia.com}\\
{\small\emph{Quantagonia}}

\textbf{Boro Sofranac}\hfill{\ttfamily boro.sofranac@quantagonia.com}\\
{\small\emph{Quantagonia}}

\vspace{3mm}

\begin{abstract}
  In this paper, we present a new method to solve a certain type of \emph{Semidefinite Programming} (SDP)
  problems. These types of SDPs naturally arise in the \emph{Quadratic Convex Reformulation} (QCR) method and can be
  used to obtain dual bounds of \emph{Quadratic Unconstrained Binary Optimization} (QUBO) problems. QUBO problems have
  recently become the focus of attention in the quantum computing and optimization communities as they are well suited
  to both gate-based and annealing-based quantum computers on one side, and can encompass an exceptional variety of
  combinatorial optimization problems on the other. Our new method can be easily \emph{warm-started}, making it very
  effective when embedded into a \emph{branch-and-bound} scheme and used to solve the QUBO problem to global
  optimality. We test our method in this setting and present computational results showing the effectiveness of our
  approach.
\end{abstract}

\section{Introduction}
\label{sec:intro}

This paper is concerned with the \emph{Quadratic Unconstrained Binary Optimization}~(QUBO) problem
\begin{equation}\label{eq:qubo}
  \max\, \bm{x}^\top Q\bm{x} + \bm{c}^\top \bm{x}, \; \bm{x} \in \{0,1\}^n,
  \tag{QUBO}
\end{equation}
where $Q \in \R^{n \times n}$ is a symmetric matrix and $\bm{c} \in \R^n$ is the linear term.
QUBO problems are known to be $\mathcal{NP}$-hard \cite{Barahona1982}.
An exceptional variety of combinatorial optimization problems can be transformed to the QUBO format, see, \eg,
\cite{Kochenberger2014,Kochenberger2006,Lucas2014} for an exhaustive number of examples.
QUBOs are currently the most widely adopted optimization model used to access quantum computers, and as such they receive significant attention of researchers both on the classical and quantum side. For a detailed discussion on the connection between QUBOs and Quantum Computing we refer the interested reader to~\cite{Glover2022}.

Most approaches to solve the QUBO problem to global optimality rely on the \emph{branch and bound} algorithm
\cite{Land1960}. The main idea of this algorithm is to split the original problem into subproblems which are easier to
solve (branching).
By repeatedly branching on the subproblems, a search tree is obtained.
At each subproblem, \emph{primal} and \emph{dual} bounds are computed.
Primal bounds correspond to objective values of \emph{feasible} solutions and in general they are computed by classic or quantum heuristics.
In contrast, dual bounds bound the objective of the \emph{optimal} solution from above (below in case of minimization).
The rationale is that the bounds lead to detecting and pruning suboptimal nodes of the tree that cannot contain improving solutions.
This is the case if the best known feasible solution has a better objective value than the bound of the
subproblem at hand.
This way, the algorithm tries to avoid having to enumerate exponentially many subproblems.

Branching typically involves splitting a domain of one variable into two subdomains. Becuase of this, subproblems close to each other in the search tree tend to be very similar. Consequently, \emph{warmstarting} plays an extremly important role in the solver used to obtain dual bounds. Warmstarting refers to the methods' ability to use the solution of one problem as a starting point for the solving of a different but similar problem. To illustrate the importance of warmstarting in branch and bounds algorithms, we note that the capability of the \emph{Dual Simplex} algorithm to be warmstarted is one of the main reason why it is still a de-facto standard method used to obtain dual bounds in MIP problems \cite{Achterberg2009}. In contrast, inherent challenges in effectively warmstarting the IPM method~\cite{Engau2011} have limited its applicability in branch and bound algorithms, in spite of its superior worst-case runtime guarantees.

In this paper, we will use the \emph{Quadratic Convex Reformulation}~(QCR) method
\cite{Hammer1970} to compute dual bounds of QUBO problems.
We will give more details on the QCR and how it bounds the QUBO problem in Section~\ref{sec:preliminaries}.
In its heart, the QCR method involves solving a certain type of a \emph{Semidefinite Programming}~(SDP) problem.
Existing QUBO bounding approaches based on the QCR method all solve the related SDP problem by off-the-shelf SDP
approaches \cite{Billionnet2006,Natarajan:2013,Billionnet2009,Billionnet2010,Krner1988,Shor1988-qt,Poljak1995-dk}. 
To the best of our knowledge, in this work we provide the first approach which exploits the special structure of this
SDP and develops a tailored method to solve it.
One of the main benefits of the proposed method is that it is well suited for \emph{warmstarting}.

\subsection{Related Work}
\label{sec:related_work}

In the literature, several streams related to computing dual bounds for QUBOs exist.
Apart from the QCR method, several authors, including \cite{Krner1988,Poljak1995-dk,Shor1988-qt} among others,
consider the ``classical'' SDP relaxation of QUBO. In fact, this SDP relaxation is related to the SDP occurring in the
QCR method, see Section~\ref{sec:preliminaries}. Another class of methods exploits a concept known as \emph{roof
  duality} to derive QUBO bounds, see \cite{Hammer-et-al:1984,Boros-et-al:2008,Mauri-Lorena:2011}. Stemming from the
equivalence between the QUBO and the well-known MaxCut problem \cite{Barahona-et-al:1989}, all methods used to bound
MaxCut problem can be used to bound QUBOs, \eg, \cite{Rendl-et-al:2010,Rehfeldt-et-al:2023,Rendl:1999,Poljak-Rendl:1995,Gusmeroli-et-al:2022,Krislock-et-al:2017,Hrga-Povh:2021}.

In terms of solving SDP problems, the current approaches can in general be divided into those based on the
\emph{interior point method},
see~\cite{Nesterov-Nemirovsky:1992,Nesterov-Nemirovsky:1994,Anstreicher:2000,Lee-Song:2020}, and those based on the
\emph{cutting plane method},
see~\cite{Shor:1977,Khachiyan:80,Khachiyan-Tarasov:1988,Vaidya:1996,Krishnan:2002,Lee-Wong:2015,Lee-Song:2020b}.

\subsection{Contribution}
\label{sec:contribution}

The contributions of this paper are as follows. We present a new method that can solve SDPs coming out the QCR method used to bound QUBOs. Our method consists of two parts: i) the transformation of the SDP problem into an equivalent constrained convex problem and ii) the solving of the transformed problem. For the first part, we present the necessary transformations and show how the two problems are connected in Section \ref{sub:qubo_plane_projection}. These transformations exploit the fact that the feasible region of the SDP in question is unbounded in the direction oposite the objective function. For the second part, in Section \ref{sub:fw_method}, we propose an algorithm tailored to our goal of solving the QUBO problem to global optimality in a branch-and-bound scheme. What makes this algorithm effective in the branch-and-bound context is that it can be easily warmstarted. 

Additionally, in Section \ref{sec:function_oracles}, we present \emph{function oracles} which can be used as building blocks to construct optimization algorithms over the transformed problem. We show how the special structure of the SDP used to bound QUBOs can be exploited to obtain efficient implementations of oracles for computing boundary points of the feasible set, evaluating the objective function at a given point and computing its gradient direction. Lastly, we present computational experiments demonstrating the effectiveness of our approach.

The rest of this paper is organized as follows. Section \ref{sec:preliminaries} gives a brief overview of the QCR method and how QUBO bounds can be extracted from it. Section \ref{sec:sdp} presents our method to solve the SDP resulting from the QCR method in the context of solving QUBOs. Lastly, in Section \ref{computational_study}, we present the computational experiemnts. 

\section{The QCR method and dual bounds for QUBOs}
\label{sec:preliminaries}

A QUBO problem is in general non-linear and non-convex. However, it is possible to reformulate it to an
equivalent convex problem. This was first observed by Hammer and Rubin \cite{Hammer1970} and later extended by
Billionnet et al.~\cite{Billionnet2006,Billionnet2009,Billionnet2010}. The resulting method is called \emph{Quadratic Convex Reformulation}. In this Section, we will provide a brief overview of this method
and how dual bounds for QUBO problems can be extracted from it.

For a given Problem~\eqref{eq:qubo}, the QCR method is based on computing a shift vector $\bm{u}\in\R^n$ such that $Q - \text{diag}\{\bm{u}\} \preceq 0$, where $\text{diag}\{\bm{u}\}$ is the diagonal matrix obtained from the vector $\bm{u}$. Then, since $x_i$ is binary, $\bm{x}^T\text{diag}\{\bm{u}\}\bm{x} = \bm{u}^T\bm{x}$. Thus, the original problem can be equivalently written as
\begin{equation}
  \label{eq:extqubo_qcr}
  \max
  \quad
  g_{\bm{u}}(\bm{x}) = \bm{x}^T(Q-\text{diag}\{\bm{u}\})\bm{x} + (\bm{c}+\bm{u})^T\bm{x}
  \quad
  \st
  \quad
  \bm{x} \in \{0,1\}^n,
\end{equation}
which is a convex QUBO, \ie, its continuous relaxation is convex.
For a given shift~$\bm{u}$ the QUBO dual bound~$\beta_{\bm{u}}$ is obtained by solving the continuous relaxation of Problem~\eqref{eq:extqubo_qcr}, \ie, $\beta_u = \max g_{\bm{u}}(\bm{x})$ subject to $\bm{x} \in [0,1]^n$.
In theory, one could use any perturbation vector $\bm{u}$ that fulfils $Q - \text{diag}\{\bm{u}\} \preceq 0$
to obtain a valid QUBO dual bound.

Billionnet et al.~\cite{Billionnet2006} present a method to compute an \emph{optimal}~$\bm{u^{\ast}}$ that leads to the tightest dual bound $\beta_{\bm{u}^{\ast}}$.
The vector~$\bm{u^{\ast}}$ is part of an optimal solution of the SDP as follows:
\begin{equation}
  \label{eq:original_sdp}
  r^* = \min_{r,\bm{u}}
  \quad
  r
  \quad
  \st
  \quad
  \begin{bmatrix}
    r  & -(\bm{c} + \bm{u})^T / 2 \\
    -(\bm{c} + \bm{u}) / 2 & \text{diag}(\bm{u}) - Q
  \end{bmatrix}
  \succeq 0.
\end{equation}

The authors of \cite{Billionnet2006} also show that $\beta_{\bm{u^{\ast}}} = r^*$ holds at optimality, which means that once we solve the SDP~\eqref{eq:original_sdp}, we directly obtain the tightest QUBO dual bound and do not have to solve the continuous relaxtion of the convexified QUBO~\eqref{eq:extqubo_qcr}.
Additionally, observe that any \emph{feasible} solution of Problem~\eqref{eq:original_sdp} leads to a valid convexification shift and hence a valid bound.
This can also be seen from the fact that the SDP~\eqref{eq:original_sdp} is the dual problem of the classical SDP relaxation of QUBO, see \cite{Billionnet2006}. Thus, the objective values of all feasible solutions to Problem~\eqref{eq:original_sdp} are valid bounds.

The remainder of the paper will now focus on solving the SDP~\eqref{eq:original_sdp}.

\section{Transforming the SDP relaxation using a plane projection}
\label{sec:sdp}

As discussed in Section \ref{sec:preliminaries}, solving the SDP~\eqref{eq:original_sdp} leads to our desired QUBO bound. In this Section, we present a new method to solve this SDP. The method consists of two parts: i) we will transform the SDP problem to a constrained convex problem and show how the optimum of the second problem directly reveals the optimum of the first. We will refer to this part of the method as the \emph{plane projection} as it will involve orthogonal projections of points from the feasible region of the SDP onto a certain hyperplane. ii) We will show how to solve the transformed problem, effectively obtaining a bound to the QUBO problem. We will denote the convex objective function of the transformed problem as $f$ (formally defined in~\eqref{eq:f_definition}). The intersection of the above-mentioned hyperplane with the feasible region of the SDP \eqref{eq:original_sdp} will form the domain of $f$.

In Section \ref{sub:qubo_plane_projection} we present the plane projection method for the SDP~\eqref{eq:original_sdp}. Then, we derive efficient computational oracles to compute the value and the gradient of $f$, as well as an oracle to compute the intersection between a ray and the boundary of the domain of $f$, as presented in Section \ref{sec:function_oracles}. Section \ref{sub:feasible_sdp_point} shows how to compute a starting feasible point for our method. These algebraic operations are necessary tools to develop a practical method to solve the transformed problem, presented in Section \ref{sub:fw_method}.

\subsection{The plane projection method for bounding a QUBO}
\label{sub:qubo_plane_projection}

The feasible region of the SDP in \eqref{eq:original_sdp} is given by the following Linear Matrix Inequality (LMI):

\begin{equation}
\label{eq:LMI}
\mathcal{S}:= \{ \bm{y} \in \mathbb{R}^{n+1}\ |\ F(\bm{y}) = A_0 + A_1 y_1 + \dots + A_{n+1} y_{n+1} \succeq 0\}  ,
\end{equation}
where
\begin{equation}
\begin{split}
& A_0 = \begin{bmatrix} 0 & -\bm{c}^T/2 \\ -\bm{c}/2 & -Q \end{bmatrix},\ A_{n+1} = \begin{bmatrix} 1 & \bm{0}_n^T \\ \bm{0}_n & 0_{n\times n} \end{bmatrix} , \\ & A_i \in\mathbb{S}^{n\times n}:\ A_i(1, i+1) = A_i(i+1, 1) = -\frac{1}{2},\ A(i+1,i+1) =1 ,\ i=1,\dots ,n .
\end{split}
\end{equation}

Note that $\bm{y}$ has one additional coordinate compared to $\bm{u}$, and the coordinates that correspond to $\bm{u}$ are exactly the same, i.e., $\bm{y} = [\bm{u}, r]$. Last, note that the objective vector of the SDP is the standard orthonormal vector $e_{n+1}$.

We will consider a hyperplane $\mathcal{H}$ which is perpendicular to $e_{n+1}$ and cuts the feasible region $\mathcal{S}$, defining an intersection with a positive volume. Then, we will define a convex function over this intersection whose optimal solution will directly lead to the optimal solution of the SDP \eqref{eq:original_sdp}. Since the projection onto $\mathcal{H}$ is trivial, to improve the presentation of our method, we define the convex problem as an $n$-dimensional problem instead of $(n+1)$-dimensional.

In particular, let $\mathcal{H} := \{[\bm{u},\hat{r}]\},\ \forall \bm{u}\in\R^n$ be a hyperplane with $\hat{r} \in \R$ such that $\mathcal{H}$ defines a non-empty intersection with $\mathcal{S}$: $\text{Vol}_{n}(\mathcal{H} \cap \mathcal{S}) >0$. For convencience, we denote the intersection as $\mathcal{D} = \mathcal{H} \cap \mathcal{S}$.

Let us now define the function
 $f_{\hat{r}}(\bm{u}): \mathcal{D} \rightarrow \R$ as:
\begin{equation}
\label{eq:f_definition}
f_{\hat{r}}(\bm{u}) = \inf \{ t \in \mathbb{R}\ |\ F([\bm{u},\hat{r}] + t e_{n+1}) \succeq 0 \}
\end{equation}

Intuitively speaking, for a given point $\bm{u} \in \mathcal{D}$, $f_{\hat{r}}(\bm{u})$ gives the negative of the algebraic value of the length of the line segment from $\mathcal{D}$ to the boundary of the feasible region $\mathcal{S}$, in the direction of the negative objective $-e_{n+1}$. Observe that the graph $\{\bm{u}, f_{\hat{r}}(\bm{u}) \ | \ \bm{u} \in \mathcal{D} \}$ is exactly the boundary of the feasible region $\mathcal{S}$ of the SDP below the hyperplane $\mathcal{H}$. Additionally, notice that we invert the sign of length in~\eqref{eq:f_definition}, i.e., $f_{\hat{r}}(\bm{u}) \le 0$ for all $\bm{u} \in \mathcal{D}$. We do this for the convenience of working with a convex, rather than a concave function, see Lemma \ref{plane_opt_is_convex}.

We can find the optimal solution $\bm{u}^*$ of $f_{\hat{r}}$ (i.e., point $\bm{u}^*$ with the maximum length of the line segment to the boundary $\mathcal{S}$, see above) by solving the following optimization problem:
\begin{equation}
  \label{eq:plane_opt_problem}
  \begin{split}
    \min &\quad f_{\hat{r}}(\bm{u}), \\
    \st &\quad
          \begin{bmatrix}
            \hat{r} & -(\bm{c} + \bm{u})^T / 2\\
            -(\bm{c} + \bm{u}) / 2 & diag(\bm{u}) - Q
          \end{bmatrix}
          \succeq 0 .
  \end{split}
\end{equation}
Observe that in \eqref{eq:plane_opt_problem}, the value $\hat{r}$ is fixed. Additionally, notice that the constraints of \eqref{eq:plane_opt_problem} exactly define $\mathcal{D}$, the domain of $f_{\hat{r}}(\bm{u})$. Now we show how the optimal solutions between the SDP and the problem in \eqref{eq:plane_opt_problem} are connected.
\begin{theorem}
Let $\hat{r}$, $f_{\hat{r}}$, and $\mathcal{D}$ be defined as in \eqref{eq:f_definition}. Then, $[\bm{u}^*, \hat{r} + f_{\hat{r}}(\bm{u}^*)] \in \R^{n+1}$ is an optimal solution of the SDP \eqref{eq:original_sdp} iff $\bm{u}^*$ is an optimal solution of the optimization problem \eqref{eq:plane_opt_problem}.
\end{theorem}
\begin{proof}
Let the optimal solution of the SDP \eqref{eq:original_sdp} be $[\bm{u}^*, r^*] \in \partial\mathcal{S}$ (optimal solutions always lie on the boundary of $\mathcal{S}$). Clearly, $\hat{r} > r^*$ must hold by definition (so that $\mathcal{H}$ has a non-empty intersection with $\mathcal{S}$, see the definition of \eqref{eq:f_definition}). First, we will prove that $[\bm{u}^*, \hat{r}] \in \mathcal{D}$. $\hat{r}$ is already in $\mathcal{H}$ by definition, so we have to show that $\bm{u}^*$ lies inside $\mathcal{S}$. Since $diag(\bm{u}^*) - Q \succeq 0$ and  $\hat{r} > r^*$, we have that
$$
\begin{bmatrix} \hat{r} & -(\bm{c} + \bm{u}^*)^T / 2\\ -(\bm{c} + \bm{u}^*) / 2 & diag(\bm{u}^*) - Q \end{bmatrix} \succeq 0 ,
$$
as increasing a single diagonal element of a PSD matrix will again yield a PSD matrix, and thus, $[\bm{u}^*, \hat{r}]\in\mathcal{D}$. In fact, this shows that the feasible region $\mathcal{S}$ of the SDP \eqref{eq:original_sdp} is unbounded in the direction opposite the objective vector $e_{n+1}$, as the claim holds for any $\hat{r} > r^*$.  This allows for a intuitive geometrical interpretation of the proof, see Figure \ref{fig:sdp}. For our method, this also shows that any value of $\hat{r}$ that leads to a non-empty intersection $\mathcal{D}$ suffices for the theorem to hold.

Second, we will prove that $\bm{u}^*$ is the optimum of \eqref{eq:plane_opt_problem}. We already showed that the projection of the optimum of \eqref{eq:original_sdp} along $e_{n+1}$ onto $\mathcal{H}$ lies in $\mathcal{D}$. To finish the proof, we need to show that there does not exist any other $u \in \mathcal{D}$ such that $f_{\hat{r}}(\bm{u}) < f_{\hat{r}}(\bm{u}^*)$, which we will do by contradition. Recall that by definition, for any point $\bm{u} \in \mathcal{D}$, the point $[\bm{u}, \hat{r} + f_{\hat{r}}(\bm{u})]$ lies on the boundary $\partial \mathcal{S}$  below the plane $\mathcal{H}$, and as the objective vector os the SDP \eqref{eq:original_sdp} is $e_{n+1}$, the $n+1$-th coordinate of this point is the objective value of SDP \eqref{eq:original_sdp}. Let there exist a $\bm{u} \in \mathcal{D}$ with $f_{\hat{r}}(\bm{u}) < f_{\hat{r}}(\bm{u}^*)$. Then, $\hat{r} + f_{\hat{r}}(\bm{u}) < \hat{r} + f_{\hat{r}}(\bm{u}^*) = r^*$. This is a contradiction because $r^*$ is the optimum of the SDP \eqref{eq:original_sdp} and there does not exist any $r = \hat{r} + f_{\hat{r}}(\bm{u})$ with a better value. The proof for the other direction, i.e. if $\bm{u}^* \in\mathcal{D}$ is an optimal solution of the problem in \eqref{eq:plane_opt_problem} then the $[\bm{u}^*, \hat{r} + f_{\hat{r}}(\bm{u}^*)] \in \R^{n+1}$ is an optimal solution of the SDP  \eqref{eq:original_sdp}, can easily be verified with the same argument.

\end{proof}

\begin{SCfigure}
\centering

\begin{tikzpicture}
    \draw (0, 0) node[inner sep=0] (a) {\includegraphics[width=0.55\textwidth]{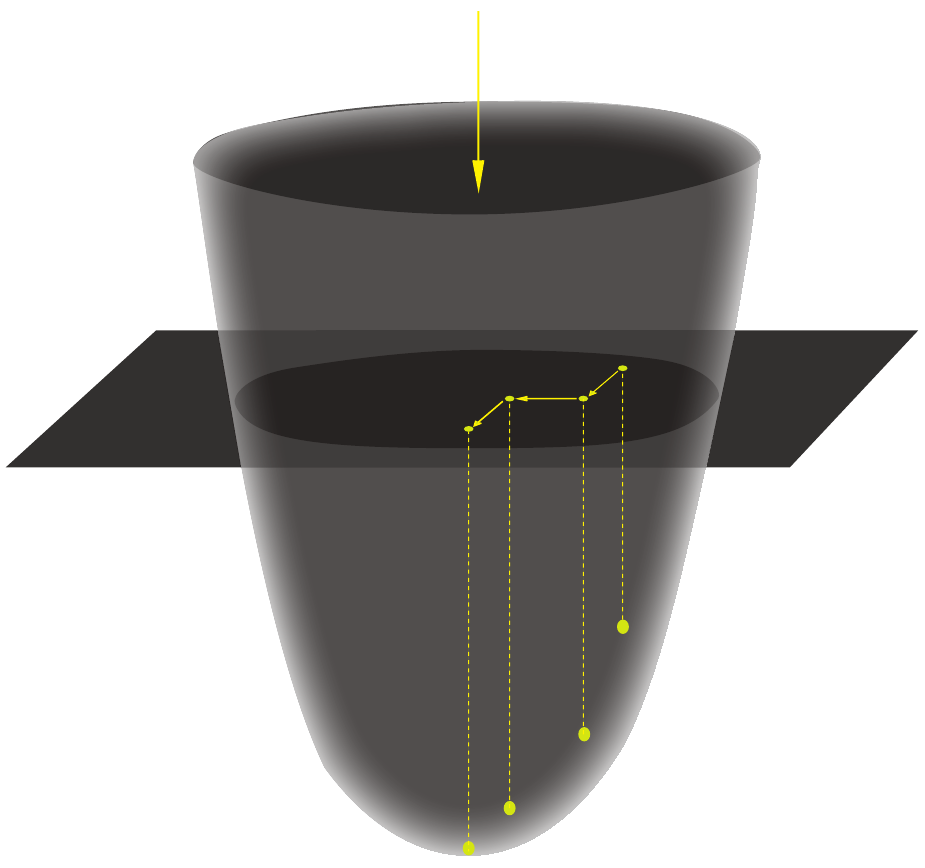}};
    \draw (1, 1) node[below left = -5.9cm and -6.2cm of a] {\textcolor{white}{$\mathcal{S}$}};
    \draw (2, 2) node[below left = -4.5cm and -3cm of a] {\textcolor{white}{$\mathcal{D}$}};
    \draw (3, 3) node[below left = -4.2cm and -1.3cm of a] {\textcolor{white}{$\mathcal{H}$}};
\end{tikzpicture}

\caption{Depiction of the feasible region $\mathcal{S}$ of the SDP \eqref{eq:original_sdp} (unbounded in the direction opposite the objective funciton), the intersection $\mathcal{D} = \mathcal{S} \cap \mathcal{H}$, where the hyperplane $\mathcal{H}$ is perpendicular to the objective vector. The set $\mathcal{D}$ is the domain of the function $f_{\hat{r}}(\cdot)$ in \eqref{eq:f_definition}. The value of $f_{\hat{r}}(\cdot)$ at a point on $\mathcal{D}$ is the negative of the algebraic value of the length of the line segment from $\mathcal{D}$ to the boundary of the feasible region $\mathcal{S}$, in the direction of the objective vector, as illustrated in the figure. The part of the boundary of $\mathcal{S}$ below $\mathcal{D}$ corresponds to the graph of $f_{\hat{r}}(\cdot)$.}
\label{fig:sdp}
\end{SCfigure}
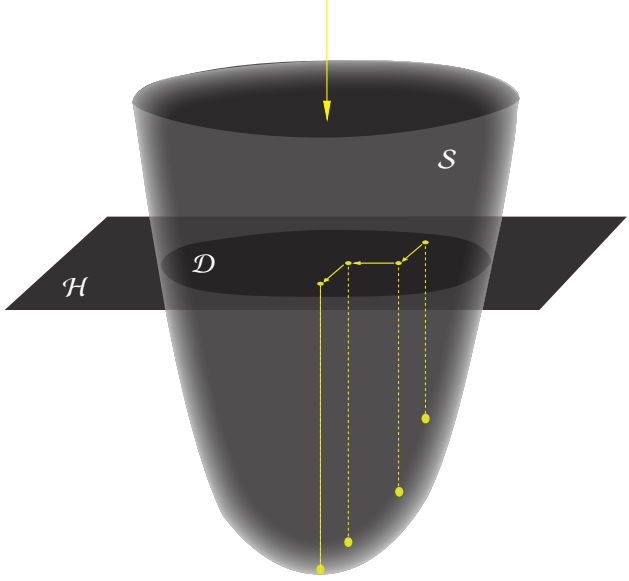

\begin{lemma}
\label{plane_opt_is_convex}
The problem \eqref{eq:plane_opt_problem} is a convex optimization problem.
\end{lemma}

\begin{proof}
 As the feasible set $\mathcal{S}$ of the SDP \eqref{eq:original_sdp} is convex, $\mathcal{D}$, an intersection of $\mathcal{S}$ with a hyperplane is trivially convex as well. As $\mathcal{D}$ is the feasible set of our problem \eqref{eq:plane_opt_problem}, we have left to show that $f_{\hat{r}}(\bm{u})$ is a convex function.  By definition, the part of the boundary of $\mathcal{S}$ below the plane $\mathcal{H}$  corresponds to the graph of $f_{\hat{r}}$. Let $[\bm{y}_1, \hat{r}], \ [\bm{y}_2, \hat{r}]
 \in \mathcal{D}$. Observe that we use the $(n+1)$-dimensional notation for $\mathcal{D}$ for the purpose of the proof here. Then, since $\mathcal{S}$ is a convex region, for all $0\leq z \leq 1$, it holds that
\begin{equation}
z [\bm{y}_1, \hat{r}] + ( 1 - z ) [\bm{y}_2, \hat{r}] \in \mathcal{D} \subset \mathcal{S}.
\end{equation}
Since for any point $[\bm{y}, \hat{r}] \in \mathcal{D}$, the point $[\bm{y}, \hat{r} + f_{\hat{r}}( \bm{y} )]$ corresponds to the point on the boundary $\partial \mathcal{S}$, see definition of \eqref{eq:f_definition}, for the convex combination of the two boundary points, we have,
\begin{equation}
\begin{split}
& z ([\bm{y}_1, \hat{r}] + f_{\hat{r}}( \bm{y}_1 )\bm{e}_{n+1}) + ( 1 - z ) ([\bm{y}_2, \hat{r}] + f_{\hat{r}}( \bm{y}_2 )\bm{e}_{n+1}) = \\ & (z [\bm{y}_1, \hat{r}] + ( 1 - z ) [\bm{y}_2, \hat{r}]) + (z f_{\hat{r}}( \bm{y}_1 ) + ( 1 - z ) f_{\hat{r}} ( \bm{y}_2 ))\bm{e}_{n+1} \in \mathcal{S},
\end{split}
\end{equation}
which follows directly from the convexity of $\mathcal{S}$. Thus, the point,
\begin{equation}\label{eq:proof_point_convex_combination}
\begin{split}
& z [\bm{y}_1, \hat{r} + f_{\hat{r}}( \bm{y}_1 )] + ( 1 - z ) [\bm{y}_2, \hat{r} + f_{\hat{r}} ( \bm{y}_2 )] = \\ & [z \bm{y}_1 + (1-z)\bm{y}_2, \hat{r} + z f_{\hat{r}}(\bm{y}_1) + (1-z)f_{\hat{r}}(\bm{y}_2)],
\end{split}
\end{equation}
belongs to the interior of $\mathcal{S}$.

By definition of \eqref{eq:f_definition}, the point
\begin{equation}\label{eq:proof_point_boundary}
[z \bm{y}_1 + (1-z)\bm{y}_2, \hat{r} + f_{\hat{r}}(z [\bm{y}_1, \hat{r}] + ( 1 - z ) [\bm{y}_2, \hat{r}])] ,
\end{equation}
lies on the boundary $\partial \mathcal{S}$. Clearly, the points $[z \bm{y}_1 + (1-z)\bm{y}_2, \hat{r}]$ and those of \eqref{eq:proof_point_convex_combination} and \eqref{eq:proof_point_boundary} lie on the line $[z \bm{y}_1 + (1-z)\bm{y}_2, \hat{r}] + te_{n+1} \in \R^{n+1}$. Because of the convexity of $\mathcal{S}$ and, since
\begin{equation}
(z [\bm{y}_1, \hat{r}] + ( 1 - z ) [\bm{y}_2, \hat{r}]) + f_{\hat{r}}(z [\bm{y}_1, \hat{r}] + ( 1 - z ) [\bm{y}_2, \hat{r}])\bm{e}_{n+1}
\end{equation}
lies on the boundary $ \partial \mathcal{S}$, we have:
\begin{equation}
f_{\hat{r}}(z [\bm{y}_1, \hat{r}] + ( 1 - z ) [\bm{y}_2, \hat{r}]) \leq z f_{\hat{r}}( \bm{y}_1 ) + ( 1 - z ) f_{\hat{r}} ( \bm{y}_2 ) .
\end{equation}
\end{proof}

\subsection{Function oracles}
\label{sec:function_oracles}

We provide three algorithms, called \textit{function oracles}: (i) an algorithm to compute the value of function $f_{\hat{r}}(\bm{u})$ for $\bm{u} \in \mathcal{D}$ (evaluation oracle), (ii) an algorithm to compute a vector proportional to the gradient, i.e., $\alpha\nabla f_{\hat{r}}(\bm{u}),\ \alpha > 0$ (gradient oracle), and (iii) an algorithm to compute the intersection of a ray starting from $\bm{u}$ with $\partial \mathcal{D}$. These oracles will form the base for obtaining a performant method to solve \eqref{eq:plane_opt_problem}.

\paragraph{Evaluation oracle.}
Notice that $f_{\hat{r}}(\bm{u}),\ \bm{u}\in\mathcal{D}$ is given by the length\footnote{Again, notice the slight abuse of terminology here: the sign of "length" is inverted so that we have negative values. cf. the definition of problem \eqref{eq:f_definition}.} of the part of the ray $\ell(t) := \{ [\bm{u}, \hat{r}] - te_{n+1}\in\R^{n+1}\ |\ t \in \R_+ \}$ in the interior of $\mathcal{S}$. Thus, we consider the ray shooting defined by $\ell(t)$ and hitting the $\partial\mathcal{S}$ at the boundary point $\bm{s}\in\R^{n+1}$. Then, $f_{\hat{r}}(\bm{u})$ is equal to the length of the segment defined by the points $[\bm{u}, \hat{r}]$ and $\bm{s}$.

Finally, as Polyak et. al. show in ~\cite{Polyak2006}, this length is given by the smallest positive eigenvalue of the following Generalized Eigenvalue Problem (GEP): $F(\bm{\ell}(t))\bm{x} =0$. In our case, we can exploit the special structure of the LMI of \eqref{eq:original_sdp} and simplify this GEP as follows:
\begin{equation}
\label{eq:GEP_eval}
\begin{split}
& F(\bm{\ell}(t))\bm{x} =0 \Rightarrow [F(\bm{y}) - t(F(\bm{e}_{n+1}) - A_0)]\bm{x} = 0 \Rightarrow \\ & (F(\bm{y}) - tA_{n+1})\bm{x} =0 \Rightarrow F(\bm{y})\bm{x} - tA_{n+1}\bm{x} = 0 ,
\end{split}
\end{equation}
where $\bm{y} = [\bm{u}, \hat{r}]$. Clearly, the GEP in~(\ref{eq:GEP_eval}) has only one real (and bounded) eigenvalue, since only one element in $A_{n+1}$ is non-zero. Thus, we can solve the following GEP: $\lambda F(\bm{y})\bm{x} - A_{n+1}\bm{x} = 0.$ We then invert its single non-zero real eigenvalue to obtain the wanted eigenvalue of the GEP in~(\ref{eq:GEP_eval}). This GEP is equivalent to the following eigenvalue problem:
\begin{equation}\label{eq:regular_eig_problem}
F(\bm{y})^{-1}A_{n+1}\bm{x} - \lambda \bm{x} = 0 .
\end{equation}
The matrix $F(\bm{y})^{-1}A_{n+1}$ is a zero matrix except for its first column, which is equal to the solution of the following linear system:
\begin{equation}
\label{eq:linear_system_eval}
F(\bm{y})\bm{z} = A_{n+1}(:,1) \Rightarrow F(\bm{y})\bm{z} = \begin{bmatrix} 1 & 0 & \ldots & 0 \end{bmatrix}^T .
\end{equation}
Finally, the wanted eigenvalue is the leftmost top (first column - first row) element of the matrix $F(\bm{y})^{-1}A_{n+1}$ which is equal to the first element of vector $\bm{z}$. Summarizing, $f_{\hat{r}}(\bm{u}) = - 1/z_1$, where $1/z_1$ is the inverse of the first element of $\bm{z}$. To compute $\bm{z}$ we have to solve the linear system in~(\ref{eq:linear_system_eval}).

\begin{remark}
To evaluate the function $f_{\hat{r}}(\bm{u})$ at a point $\bm{u} \in \mathcal{D}$ it takes $O(n^{\omega})$ operations, where $\omega$ is the \emph{matrix multiplication constant}, $\omega \approx 2.37$.
\end{remark}

\paragraph{Gradient oracle.}

Recall that  $f_{\hat{r}}$ is a function on $\R^n$, while the boundary of $\mathcal{S}$ lies in $\R^{n+1}$. Also, recall that the graph of function $f_{\hat{r}}$ is exactly the boundary of $\mathcal{S}$ lying below the hyperplane $\mathcal{H}$. Thus, $\nabla f_{\hat{r}}(\bm{u})$  is equal to the normal vector of the tangent plane on the boundary point $[\bm{u}, \hat{r} + f_{\hat{r}}(\bm{u})]$, projected (orthogonally) down onto $\mathcal{D}$.

Since $\partial\mathcal{S}$ is given by $\text{det}F(\bm{y}) = 0,\ \bm{y}\in\R^{n+1}$, the gradient $\nabla f_{\hat{r}}(\bm{u})$ is the first $n$ coordinates of the gradient $\nabla \text{det}F([\bm{u}, \hat{r} + f_{\hat{r}}(\bm{u})])$. From \cite{Chalkis:2022} we get that,
\begin{equation}
\label{eq:gradient}
\begin{split}
& (\bm{x}^TA_1\bm{x}, \dots, \bm{x}^TA_{n+1}\bm{x})
= \alpha \nabla \text{det}F([\bm{u}, \hat{r} + f_{r^*}(\bm{u})]) = \\ & \alpha[\nabla f_{\hat{r}}(\bm{u}), \bm{x}^TA_{n+1}\bm{x}] ,\ \alpha>0 ,
\end{split}
\end{equation}
where $\bm{x}$ is any non-trivial vector from the kernel of $F([\bm{u}, \hat{r} + f_{\hat{r}}(\bm{u})])$. It turns out that the solution vector of \eqref{eq:linear_system_eval} satisfies this condition, as shown in   Lemma~\ref{lem:eigenvector}:

\begin{lemma}{}
\label{lem:eigenvector}
The solution vector $\bm{z} \in \R^{n+1}$ of the linear system in~(\ref{eq:linear_system_eval}) is a non-trivial vector from the kernel of $F([\bm{u}, \hat{r} + f_{\hat{r}}(\bm{u})])$.
\end{lemma}

\begin{proof}
Since $f_{\hat{r}}(\bm{u}) = 1/z_1$ is equal to the smallest positive eigenvalue of the GEP in~(\ref{eq:GEP_eval}) we have that,
\begin{equation}
\label{eigenvector_proof}
\begin{split}
& F(\bm{y})\bm{x} + f_{\hat{r}}(\bm{u})A_{n+1}\bm{x} = 0 \Rightarrow (F([\bm{u}, \hat{r}]) + f_{\hat{r}}(\bm{u})A_{n+1})\bm{x} = 0 \Rightarrow \\ & F([\bm{u}, \hat{r} + f_{\hat{r}}(\bm{u})])\bm{x} = 0 .
\end{split}
\end{equation}
Thus, the corresponding eigenvector of the eigenvalue $f_{\hat{r}}(\bm{u})$ is a non-trivial vector from the kernel of $F([\bm{u}, \hat{r} + f_{\hat{r}}(\bm{u})])$ and we can use it to compute the gradient in~(\ref{eq:gradient}).

Now we prove that this eigenvector is the solution vector $\bm{z}$ of the linear system in~(\ref{eq:linear_system_eval}).
According to~\eqref{eq:regular_eig_problem} the wanted eigenvector is identical to the eigenvector that corresponds to the sole non-zero eigenvalue of the matrix $W = -F(\bm{y})^{-1}A_{n+1}$. This eigenvalue is $z_1$, where $F(\bm{y})\bm{z} = -A_{n+1}(:,1)$. Now, notice that the first column of $W$ is equal to the vector $\bm{z}$, while the rest of its elements are equal to zero. Therefore, $W\bm{z} = z_1\bm{z}$ holds, and thus, the vector $\bm{z}$ is the eigenvector that corresponds to the eigenvalue $z_1$.
\end{proof}

\begin{remark}
Given the solution of the linear system in~(\ref{eq:linear_system_eval}), to compute a vector proportional to the gradient of $f_{\hat{r}}(\bm{u})$ takes $O(n)$ operations since each coordinate $\bm{x}^TA_i\bm{x},\ i\in[n]$ can be computed in $O(1)$ operations.
\end{remark}
One could also compute the actual $\nabla f_{\hat{r}}(\bm{u})$ following~\cite{Chalkis:2022}. However, that would be computationally more expensive, since it requires to compute all the eigenvalues of the matrix $F([\bm{u}, \hat{r} + f_{\hat{r}}(\bm{u})])$, while our method will not require to compute the actual gradient, see Section \ref{sub:fw_method}.

\paragraph{Boundary oracle.}
Let the ray $\ell(t) := \{[\bm{u}, \hat{r}] + t [\bm{d}, 0]\ |\ t\in\R_+\}$, where $\bm{d}\in\R^n$ gives the direction of the ray in $\mathcal{D}$.
To obtain the intersection $\ell(t) \cap \partial\mathcal{D}$, we need to compute the smallest positive eigenvalue $\lambda^*$ of the following GEP,

\begin{equation}
\label{eq:GEP_intersection}
\begin{split}
& F(\bm{\ell}(t))\bm{x} =0 \Rightarrow F([\bm{u}, \hat{r}] + t[\bm{d},0])\bm{x} = 0 \Rightarrow \\ & (F([\bm{u}, \hat{r}]) + t[F([\bm{d},0]) - A_0])\bm{x} = 0 \Rightarrow (C_1 + t C_2)\bm{x} = 0 ,
\end{split}
\end{equation}
where,
\begin{equation}
C_1 = \begin{bmatrix} \hat{r} & -(\bm{c} + \bm{u})^T / 2\\ -(\bm{c} + \bm{u}) / 2 & diag(\bm{u}) - Q \end{bmatrix},\
C_2 = \begin{bmatrix} 0 & \bm{d}^T / 2\\ \bm{d} / 2 & diag(\bm{d}) \end{bmatrix} .
\end{equation}
Then, the intersection point can be computed easily as, $[\bm{u}, \hat{r}] + \lambda^*[\bm{d}, 0] = \ell(t) \cap \partial\mathcal{S}$.

\begin{lemma}[From~\cite{Chalkis:2022}]
The arithmetic complexity to compute the intersection $\ell(t) \cap \partial\mathcal{S}$ up to precision $\epsilon>0$ is $O(n^\omega + n\log(1/\epsilon))$.
\end{lemma}

In section~\ref{computational_study} we provide an efficient algorithm based on Lanczos method to compute the wanted eigenvalue regarding practical performance.

\subsection{Computing a feasible SDP point}
\label{sub:feasible_sdp_point}

Our methods require a starting feasible point of the convex problem~(\ref{eq:plane_opt_problem}) to start computations. Notice that this point would also be a feasible point of the SDP~(\ref{eq:original_sdp}).
To compute such a feasible point we first compute a trivial convexification vector, that is, $\overline{\bm{u}} = \lambda_{\max}(Q)\bm{1}_{n}$, where $\bm{1}_{n}$ is an $n$-dimensional vector with all coordinates equal to $1$. Since $diag(\overline{\bm{u}}) - Q \succeq 0$, we need to compute a scalar $\hat{r}$ such that $F([\overline{\bm{u}}, \hat{r}]) \succ 0$.

To this end, we set the initial $\bm{y}$ to be equal to $\overline{\bm{u}}$ in the corresponding coordinates, and leave the $(n+1)$th coordinate to any random value $r_d$. To compute a feasible scalar $\hat{r}$ we consider the intersection between the ray
$\bm{\ell}(t) = \{\bm{y} - t\bm{e}_{n+1} \} ,\ t\in\R_+$,
and the boundary of the feasible region $\mathcal{S}$. This is given by the smallest positive eigenvalue of the following GEP:
\begin{equation}
\begin{split}
& F(\bm{\ell}(t))\bm{x} = 0 \Rightarrow [F(\bm{y}) - t(F(\bm{e}_{n+1}) - A_0)]\bm{x} = 0  \Rightarrow F(\bm{y})\bm{x} - tA_{n+1}\bm{x} = 0 .
\end{split}
\end{equation}
Thus, following the methodology in Section~\ref{sec:function_oracles}, we solve the linear system in~(\ref{eq:linear_system_eval}) and invert the first coordinate of its solution $\bm{z}$. Summarizing, $\hat{r} = r_d - 1/z_1$, and since this value leads us to the boundary of the feasible region $\mathcal{S}$, we set $\hat{r} \leftarrow (r_d + 1/z_1)(1+\epsilon),\ \epsilon > 0$, to obtain a point in the interior. Notice that the approach above computes a feasible point of the SDP \eqref{eq:original_sdp} for any vector $\bm{u}\in\R^n$ s.t.\ $diag(\bm{u}) - Q \succeq 0$

\subsection{Solving the transformed problem}
\label{sub:fw_method}

In this section, we will present an algorithm to solve the transformed problem \eqref{eq:plane_opt_problem}. Observe that as this is a constrained convex optimization problem, any method targeting such problems would suffice to solve it, e.g., the constrained coordinate descend algoritm \cite{nesterov_coordinate_descent}, the Frank-Wolfe algorithm \cite{fw56}, etc. Both mentioned methods come with worst-case convergence guarantees, meaning that at convergence we would obtain the optimal solution of \eqref{eq:plane_opt_problem}. Depending on the application area, some methods might be more suitable than others. In this paper, we will showcase the proposed methods by solving the QUBO problem in a branch-and-bound scheme, where we solve \eqref{eq:plane_opt_problem} at each node of the search tree to obtain a valid bound for the node. Depending on the instance at hand, it can be beneficial for the algorithm to obtain \emph{weaker} bounds \emph{faster} at a given node and branch (hope here is that the children will be easier to solve), rather than spend a lot of time obtainin a strong bound for the node. Observe that by allowing the algorithm to use weaker but valid bounds in the nodes, it will still converge to global optimality as in the worst case it callapses to enumerating all possible solutions. In conclusion, for our case, the determining factor for performance of the overall algorithm will be the trade-off between the quality of the solution of \eqref{eq:plane_opt_problem} and the time to obtain it. The algorithm we present reflects this, as it does not necessarily produce an optimal solution of \eqref{eq:plane_opt_problem} and will stop early if it detects slow progress, allowing the branch-and-bound scheme to branch and look for easier subproblems.

Lastly, before presenting our algorithm, we mention that in the case of the constrained coordinate descent algorithm from \cite{nesterov_coordinate_descent}, we were able to design an algorithm  that optimally solves the problem \eqref{eq:plane_opt_problem} and has a better worst case complexity compared to the state-of-the-art SDP algortihm in~\cite{Lee-Song:2020} when it is applied to solve the QUBO's dual SDP relaxation of \eqref{eq:original_sdp} by exploiting the special structure of our problem, in a similar way we were able to simplify the function oracles presented in Section \ref{sec:function_oracles}. However, the resulting algorithm still significantly underperformed in our experiments compared to the algorithm presented in this Section. Similarly, for the Frank-Wolfe algorithm, the \emph{linear minimization oracle} is a SDP problem itself, making it too expensive for our setting (see \cite{Jaggi:2013} for details on the Frank-Wolfe algorithm). We do not further consider these methods for this conference submission due to space restrictions, but we plan on presenting the reductions and computational results for these methods as well in the extended version of this paper.

Our method is presented in Alg.~\ref{alg:descent_method}.
\begin{algorithm}[t]
  \caption{}
\label{alg:descent_method}
\begin{algorithmic}[1]
  \REQUIRE Point $[\bm{u}, \hat{r}] \in \mathcal{S}$, iteration limits $N,k_1,k_2 \in \mathbb{N}_+$.
  \ENSURE Upper bound $\hat{r} + f_{\hat{r}}(\hat{\bm{u}})$ for the \eqref{eq:qubo} where $\hat{\bm{u}}$ is a solution of \eqref{eq:plane_opt_problem}.
  \STATE $i\leftarrow 1;\ j\leftarrow 0$
  \WHILE{$i \leq N$}
    \STATE $d \leftarrow \nabla f_{\hat{r}}(\bm{u})\ /\ \| \nabla f_{\hat{r}}(\bm{u}) \|$ \label{normalized_gradient}
    \STATE $\bm{s} \leftarrow \bm{u} - t \nabla f_{\hat{r}}(\bm{u}),\ t= \sup\{ t\in\R_+\ |\ F([\bm{u}, \hat{r}] - t[\nabla f_{\hat{r}}(\bm{u}), 0]) \succeq 0 )\}$ \label{line_with_bnd_oracle}
    \STATE $q \leftarrow 1;$ $\bm{u}_{+} \leftarrow \bm{s};$ $\bm{u}_{-} \leftarrow \bm{u};$ $boundary \leftarrow \text{ true}$
    \WHILE{$q \leq k_1$}
      \STATE $\bar{\bm{u}} \leftarrow (\bm{u}_{+} + \bm{u}_{-}) / 2$
      \IF{$\langle \nabla f_{\hat{r}}(\bar{\bm{u}}), (\bm{u}_{+} - \bm{u}_{-}) \rangle > 0$} \label{line_with_binary_search_start}
        \STATE $\bm{u}_{+} \leftarrow \bar{\bm{u}}$; $boundary \leftarrow \text{ false}$
      \ELSE
        \STATE $\bm{u}_{-}\leftarrow \bar{\bm{u}}$
      \ENDIF  \label{line_with_binary_search_end}
      \STATE $q \leftarrow q + 1$
    \ENDWHILE
    \IF{$boundary$}
      \STATE $j \leftarrow j + 1$
      \IF{$j == k_2$} \label{boundary_stopping_criterion}
        \STATE \textit{break}
      \ENDIF
    \ENDIF
    \STATE $i \leftarrow i+1$
  \ENDWHILE
  \STATE $\hat{\bm{u}} \leftarrow \bm{u}$; $r \leftarrow \hat{r} + f_{\hat{r}}(\hat{\bm{u}})$
  \RETURN{$\hat{\bm{u}},\ r$}
  \end{algorithmic}
\end{algorithm}
It takes as input a feasible point of the problem \eqref{eq:plane_opt_problem} and outputs a valid upper bound for \eqref{eq:qubo}. It employs gradient directions to improve the solution in each iteration, with a limit of $N$ iterations. Its implementation requires the function oracles defined in Section~\ref{sec:function_oracles}. Simply, in each step, our method computes the normalized gradient at the current point $\bm{u}$, see Line \ref{normalized_gradient}. It shoots a ray starting from $\bm{u}$ towards the negative gradient direction and computes the boundary point $\bm{s}$ where the ray hits the boundary $\partial\mathcal{D}$, using the boundary oracle presented in Section~\ref{sec:function_oracles}, see Line \ref{line_with_bnd_oracle}.

Then, our method applies a search to approximate the best point according to $f_{\hat{r}}$ on the segment defined by $\bm{u}$ and $\bm{s}$ , i.e. it computes,
\begin{equation}
\bar{\bm{u}} := \argmin\limits_{\bm{x} \in \{t\bm{u} + (1-t)\bm{s}\}}f_{\hat{r}}(\bm{x}),\ 0 \leq t \leq 1 .
\end{equation}
To approximate $\bar{\bm{u}}$ we provide a bisection method applied on the segment $\{t\bm{u} + (1-t)\bm{s},\ 0 \leq t \leq 1$. Let $\bm{u}_{-},\ \bm{u}_{+} \in \mathcal{D}$ the extreme points of the segment in each iteration of the bisection method. First, we project the normalized gradient $\frac{\nabla f_{\hat{r}}(\bm{u})}{\| \nabla f_{\hat{r}}(\bm{u}) \|}$ onto the segment defined by the extreme points $\bm{u}_{-},\ \bm{u}_{+}$, say,
\begin{equation}
\bm{d} = \frac{\langle \nabla f_{\hat{r}}(\bm{u}), (\bm{u}_{+} - \bm{u}_{-}) \rangle}{\| \nabla f_{\hat{r}}(\bm{u}) \| \| \bm{u}_{+} - \bm{u}_{-} \|^2} (\bm{u}_{+} - \bm{u}_{-}) .
\end{equation}
Then, we employ the innner product $\langle \bm{d}, (\bm{u}_{+} - \bm{u}_{-}) \rangle$ to decide in which sub-interval between $\{\bm{u}_-,\bm{u}\}$ and $\{\bm{u},\bm{u}_+\}$ to continue searching. In particular, if $\langle \nabla f_{\hat{r}}(\bm{u}), (\bm{u}_{+} - \bm{u}_{-}) \rangle > 0$, we continue within $\{\bm{u}_-,\bm{u}\}$, otherwise we continue within $\{\bm{u},\bm{u}_+\}$, see Lines \ref{line_with_binary_search_start} through \ref{line_with_binary_search_end}. The bisection iteration limit $k_1$ allows us to stop when the search segments become small enough.

Lastly, in addition to the global iteration limit $N$, our method also features a stopping criterion tailored for the edge-case when it converges to a point on the boundary of $\mathcal{D}$. It will stop the algorithm if it detects that for $k_2$ consecutive iterations (search directions), the bisection method did not once move away from the boundary $\partial \mathcal{D}$ (the bisection method \emph{approximated} the boundary point $s$), see Line \ref{boundary_stopping_criterion}. This criterion prevents our method from getting stuck on the boundary $\partial\mathcal{D}$ as it terminates the method when the current iterate is very close to the boundary and the gradient direction can not move it away from it.

\section{Computational study}
\label{computational_study}

In this section, we present a computational analysis of Algorithm~\ref{alg:descent_method}.
First, we give some details on the computational setup, before we highlight the efficiency of Algorithm~\ref{alg:descent_method} to solve the SDP~\eqref{eq:original_sdp} when a warmstart is given. We also show that this is in contrast to the state-of-the-art SDP solvers \SDPA and \DSDP, which lack this feature.
After that, we integrate Algorithm~\ref{alg:descent_method} as a dual solver into a branch-and-bound framework to solve QUBOs to optimality and conduct a comparative analysis against both \SDPA and \DSDP as the dual solvers.

\subsection{General computational setup}
\label{sec:comp-setup}

All our experiments were carried out on a PC with a 11th Gen Intel® Core™ i7-1165G7 CPU with 8 cores, \SI{2.8}{\giga\hertz}, and \SI{32}{\giga\byte} RAM running \textsf{Ubuntu 20}. We implemented Algorithm~\ref{alg:descent_method} using \textsf{C++-11}.

In order to compute the smallest positive eigenvalue of the GEP in~\eqref{eq:GEP_intersection} we adopt the most efficient practical approach according to~\cite{Helmberg-et-al:1996,Chalkis:2022,Yamashita:2003}. That is, we compute the largest eigenvalue of the GEP: $(C_2 + t C_1)\bm{x} = 0$, and invert the computed eigenvalue. Since it holds that $C_1 \succeq 0$, one could transform this GEP into a regular symmetric eigenvalue problem by employing the cholesky factorization of $C_1 = L^TL$. Then, the GEP becomes equivalent to $(B - t I)\bm{x} = 0$, where $B=-L^{-T}C_2L^{-1}$. Since this is a symmetric eigenvalue problem, its maximum eigenvalue can be efficiently computed using a power method, \eg, via the Lanczos algorithm. With this approach the eigenvalue can be computed after $O(n^3)$ operations. In our implementation, we use the Lanczos algorithm implemented in the \textsf{C++} library \Spectra.
Further parameterizations and implementation details of Algorithm~\ref{alg:descent_method} are experiment-specific and are thus specified in the following two subsections.

For our computational comparisons, we also linked \DSDP~\cite{dsdp1,dsdp2,dsdp3} and \SDPA~7~\cite{Yamashita:2010} and used both with its default settings.

Overall, we used 129~instances from three publicly available instance collections for our evaluations.
We provide a reference per subset and details on the number of instances, number of variables, and densities in Table~\ref{tab:testset}.
\begin{table}[t]
  \centering
  \caption{References and sizes of our test set.}
  \label{tab:testset}
  \begin{tabular}{l*{5}{c}}
    \toprule
    test set & reference & \#instances & \#variables & density \\
    \midrule
    \textsf{Gset}   & \cite{Gset}  & 24  &  800 - 14000 & 0.00 - 0.06\\
    \textsf{BQP}    & \cite{Beasley:1998} & 60 & 50-2500 & 0.09 - 0.11 \\
    \textsf{BQPGKA}       & \cite{Glover:1998} & 45 & 20 - 500 & 0.07 - 0.99 \\
    \bottomrule
  \end{tabular}
\end{table}
The test set is diverse in terms of instance sizes and density.
While the \textsf{Gset} instances are rather large in terms of the number of variables, all instances are very sparse.
In contrast, the instances contained in \textsf{BQPGKA} are small to mid-sized but cover a wide range from sparse to very dense.
The \textsf{BQP} instances are sparse and range from small to large.
Note that we aimed for a test set that is balanced \wrt the number of solvable and unsolvable instances. This allows to analyze both runtimes to optimality and final relative gaps when running into the time limit.
Since all of the instances in \textsf{Gset} are challenging, we randomly selected a subset of 24~instances from the 71~instances contained in \textsf{Gset}.

\subsection{Warmstarting ability}%
\label{sec:warmstart_experiments}%
In this Section, we focus on showing the ability of our method to be warmstarted, in contrast to the interior point methods in \DSDP and \SDPA. To this end, we will intialize the three methods once with a solution far away from the optimum (\emph{coldstart}) and once with a solution close to the optimum (\emph{warmstart}), and we will record how the performance of the three methods changes.

For these experiments, we paremeterized Algorithm~\ref{alg:descent_method} as follows: we set $N=\num{50000}$, $k_1=5$, and $k_2=2$. For the coldstart, we initialize the methods with the vector $\bm{u}=c\lambda_{\max}(Q)\mathbf{1}_n$, where $c>1.05$ is a factor we compute using a binary search such that the relative gap of the objective value of the coldstart solution to the optimal objective value lies between $\si{85}{\%}$ and $\si{95}{\%}$. If the initial gap is larger than $\si{95}{\%}$ for $c=1.05$, we keep that solution. This is the case for all the instances in \textsf{Gset}.
For the warmstart, we perturb the optimal solution $\bm{u}^*$ by setting $\bm{u}_i = \bm{u}^*_i + w\rho$, where $\rho$ is a uniformly distributed number in $[w_1 10^{Y-1}, w_2 10^Y]$, $Y$ represents the order of magnitude of $\bm{u}^*_i$ and $w_1,w_2\in\R_+$ are factors such that the relative gap of the warmstart to the optimal objective value lies between $7\%$ and $8\%$.

Table~\ref{tab:warmstart_experiments} gives the arithmetic mean of the number of iterations and the runtime for each of the three methods for the instances from Table~\ref{tab:testset} and for both variants, cold- and warmstarted.
For Algorithm~\ref{alg:descent_method}, we also give the arithmetic mean of the relative gap of the output solution compared to the optimal solution. Note that \SDPA and \DSDP both compute an optimal solution up to some tolerance.
\begin{table}[t]
  \centering
  \caption{Computational results of the three methods for both a cold- and a warmstart of the SDP relaxation \eqref{eq:original_sdp} of the instances in Table~\ref{tab:testset}. \emph{StartGap} stands for the average relative gap to the optimal value of the starting point,
  \emph{Iters} for the average number of iterations until termination, \emph{Time} for the average runtime, and \emph{SolGap} for the arithmetic mean of the relative gap to the optimal value of the computed solution.}
  \label{tab:warmstart_experiments}
  \begin{tabular}{lrrrrrrr{c}}
    \toprule
    \multicolumn{2}{c}{}  & \multicolumn{3}{c}{Algorithm~\ref{alg:descent_method}} & \multicolumn{2}{c}{SDPA} & \multicolumn{2}{c}{DSDP} \\
    \cmidrule(lr){3-5}
    \cmidrule(lr){6-7}
    \cmidrule(lr){8-9}
                          & StartGap & Iters & Time & SolGap & Iters & Time & Iters & Time \\ \midrule

    \multirow{2}{*}{\textsf{BQPGKA}}  & 91.4$\%$ & 72.2 & 0.16 & 2.5$\%$ & 28.9 & 20.7  & 70.6 & 66.2 \\  
                          & 7.4$\%$ & 6.1 & 0.03 & 1.6$\%$ & 28.9 & 20.3  & 68.2 &  65.0 \\

    \multirow{2}{*}{\textsf{BQP}}  & 88.3$\%$ & 43.7 & 9.8 & 2.0$\%$ & 30.6   & 18.5 & 104.3 & 213.6 \\ 
                          &  7.8$\%$ &  4.9 & 1.1 & 2.4$\%$ & 30.6  & 18.5 & 67.0 & 202.5 \\

    \multirow{2}{*}{\textsf{Gset}} & 280$\%$ & 2.0 & 1.1 & 232$\%$ & 32.8 & 232.0  & 158.5 & 620.1 \\
                          & 7.5$\%$ & 2.0 & 1.2 & 2.3$\%$ & 32.7 & 229.1 & 158.8 & 615.0 \\
    \bottomrule
  \end{tabular}
\end{table}
It is noteworthy that both \SDPA and \DSDP require a similar number of iterations and runtime to solve the SDP relaxation no matter if the solve is cold- or warmstarted. This is expected, since both methods cannot exploit a warmstart. The only exception to this are the instances in \textsf{BQP}, for which \DSDP needs significantly less iterations when a warmstart is given. However, the runtime is improved only by around~$\SI{5}{\%}$. On the contrary, our method exhibits significantly fewer iterations and smaller runtimes when initiated with a warmstart, achieving better or similar relative accuracy. In all cases, our method terminates due to the boundary termination criterion of Line \ref{boundary_stopping_criterion} in Algorithm~\ref{alg:descent_method}, which indicates a stoppage when the method becomes confined to the boundary. Further, for all instances in \textsf{Gset}, Algorithm~\ref{alg:descent_method} consistently terminates after 2~iterations, independent of a cold- or a warmstart.
This consistent termination pattern suggests that the instances pose considerable challenges for our method.
However, when initialized with a warmstart, Algorithm~\ref{alg:descent_method} achieves an error comparable to that it achieves in the other two test sets. As expected, our method is not as accurate as \SDPA and \DSDP. Nevertheless, its ability to rapidly enhance the input solution and attain a valid QUBO bound suggests applicability in improving branch-and-bound methods for solving QUBO to optimality.

\subsection{Performance evaluation within a QUBO branch-and-bound approach}
\label{sec:comp-results-details}
In this section, we evaluate how our proposed method performs inside a QCR-based branch-and-bound approach for solving QUBOs to global optimality, \ie, we equip a branch-and-bound framework with Algorithm~\ref{alg:descent_method} to solve all bounding problems~\eqref{eq:original_sdp}.
In order to focus our analysis solely on the dual side, we implemented a branch-and-bound algorithm with the following properties:
\begin{itemize}
  \item We neither use presolve nor primal heuristics.
        Instead, we inject the best primal solution value known to us at the root node. Hence, at each node, only the dual solver is run.
  \item We use a basic but deterministic branching rule, which always branches on the free binary variable with
        the minimum index.
  \item We always select the problem with the weakest dual bound from the heap as the next node to be processed.
\end{itemize}
We realized this branch-and-bound implementation by using the \textsf{C++} branch-and-bound framework \PEBBL; see \cite{Eckstein-et-al:2015}.
For the branch-and-bound experiment, we paremeterized Algorithm~\ref{alg:descent_method} as in Section~\ref{sec:warmstart_experiments} except for the maximum number of iterations which we set to $N=5$.
The rationale is that due to the warmstarting ability of our method, in most cases a few iterations are enough to improve the dual bound enough. Preliminary experiments confirmed that this setting is the best trade-off between node throughput and solving each node to optimality to obtain a bound as tight as possible; see also the computational results in the previous subsection.
We also set the scale factor that is used to push a point on the boundary of the feasible region~$\mathcal{S}$ into the interior of $\mathcal{S}$ to $(1+\epsilon) = 1.1$; see also Section~\ref{sub:feasible_sdp_point}.
Since our method depends a lot on being warmstarted, we always solve the root node SDP with \SDPA, because it is faster than our method started from a coldstart. All subsequent bounding SDPs are solved by Algorithm~\ref{alg:descent_method}.
In each node, we project the SDP solution of the parent problem to the feasible space of the node problem to obtain a starting solution.

To the best of our knowledge, our method is the first that specifically targets the SDP~\eqref{eq:original_sdp}. Hence, as a reference for comparison, we adopt the approach taken by existing literature on solving QUBOs via the QCR method and use the same branch-and-bound framework as detailed above but equipped with a state-of-the-art, off-the-shelf SDP solver. We test the two widely used open-source solvers \DSDP and \SDPA, with default settings. In the following, we compare the three configurations on the instances as specified in Table~\ref{tab:testset}, each which a time limit of \SI{3600}{s}.

For 47 out of the 129~instances, all configuration find the optimal solution. For this subset of instances, we report shifted geometric means (with a shift of 10) of runtimes as well as minimum, maximum, mean, and median runtimes and branch-and-bound nodes in Table~\ref{tab:sgm-runtimes-solved-by-all}.
\begin{table}[t]
  \centering
  \caption{Shifted geometric means of runtimes, and statistics on runtimes and number of branch-and-bound nodes for the 47 instances that all configurations solve.}
  \label{tab:sgm-runtimes-solved-by-all}
  \resizebox{\textwidth}{!}{
  \begin{tabular}{lrrrrrrrrrr{c}}
    \toprule
    &  \multicolumn{2}{c}{shifted geo. means} & \multicolumn{4}{c}{runtimes in seconds} & \multicolumn{4}{c}{branch-and-bound nodes} \\
    \cmidrule(lr){2-3}
    \cmidrule(lr){4-7}
    \cmidrule(lr){8-11}
                        & unscaled  & scaled  & min. & max. & mean & median & min. & max. & mean & median \\
    \midrule
    \DSDP         & 69.52 & 5.00 & 0.31  & 3033.23 & 347.20 & 33.98 & 21 & 64205 & 7339.21 & 1871 \\
    \SDPA         & 40.71 & 2.93 & 0.11  & 3441.03 & 333.06 & 13.27 & 3 & 311547 & 14483.94 & 579 \\
    Algorithm~\ref{alg:descent_method}       & 13.89 & 1.00 & 0.06  & 829.92  & 53.90 &   2.54 & 47 & 445247 & 35557.34 & 3735 \\
    \bottomrule
  \end{tabular}}
\end{table}
We observe that \DSDP is entirely dominated by \SDPA, so we will not further comment on \DSDP in the following due to space restrictions of the conference submission.
Further, we observe that Algorithm~\ref{alg:descent_method} outperforms \SDPA significantly by a factor of 2.93 for the shifted geometric mean.
This is supported by the mean and median runtimes, which show speedups of around 6 and 5 times compared to \SDPA in the mean and median.
We also observe that Algorithm~\ref{alg:descent_method} requires more branch-and-bound nodes than the other two configurations. As discussed in Section~\ref{sub:fw_method}, this is due to the fact that our method focuses more on node throughput than on solving each node to optimality. This is even more pronounced due to the maximum number of iterations, which we set to $N=5$ in this experiment.

Apart from the instances that all configurations solve, there are eight additional instances that at least one solver solves.
There are three instances that are only solved by \SDPA and Algorithm~\ref{alg:descent_method} with mean runtimes of \SI{1368.62}{\second} and \SI{530.78}{\second}, respectively.
In comparison to \SDPA, our approach solves four additional instances with a mean runtime of \SI{453}{\second}. \SDPA reports a mean relative gap of \SI{1.89}{\percent} for these instances.
\SDPA also solves one instance (in \SI{569}{\second}) that our approach cannot solve. In fact, our solver runs into memory issues for this instance after \SI{1851}{\second} with a relative gap of \SI{2.54}{\percent}. Note that we treat this instance as it would have run into the time limit with the final relative gap of of \SI{2.54}{\percent} in our computational study.

The remaining 74 instances remain unsolved by all branch-and-bound configurations.
For these instances, we provide minimum, maximum, mean, and median relative gaps in Table~\ref{tab:gaps-solved-none}.
We observe that on this subset of instances, \SDPA and our method are quite comparable. While our method performs slightly better \wrt the mean, \SDPA performs marginally better in terms of median relative gaps.
\begin{table}[t]
  \centering
  \caption{Metrics based on relative gaps for the 74 instances, for which all configurations run into the time limit.}
  \label{tab:gaps-solved-none}
  \begin{tabular}{lrrrr{c}}
    \toprule
    & \multicolumn{4}{c}{relative gap in \%} \\
    \cmidrule(lr){2-5}
                         & minimum & maximum & mean & median \\
    \midrule
    \DSDP          & 1.84 & 180255.94& 5509.02 & 8.74 \\
    \SDPA          & 0.92 & 151.95 & 11.79 & 8.58 \\
    Algorithm~\ref{alg:descent_method}   & 1.79 &  64.94 & 10.82 & 8.89 \\
    \bottomrule
  \end{tabular}
\end{table}
Overall, the branch-and-bound configuration with our method as the dual solver can be considered the best performing configuration in our analysis. It solves the most instances and solves instances that all configurations solve significantly faster than the competing configurations.
For the instances for which all configurations run into the time limit, all configurations perform equally well.

In a full-length paper, we aim to further improve our method and to include a more detailed computational study.
Possible steps into this direction are to tailor subroutines of our method to sparse instances, perform a thorough parameter tuning of our method, and compare it to both other exact methods as well as general-purpose heuristics for solving the bounding SDPs.

\bibliographystyle{abbrv}
\bibliography{refs}
\end{document}